\documentclass[11 pt]{amsart}
\usepackage{latexsym,amscd,amssymb, graphicx, amsthm, bm}  
\usepackage{young}
\usepackage{arydshln}
\usepackage{tikz}

\usepackage[margin=1in]{geometry}

\numberwithin{equation}{section}

\newtheorem{theorem}{Theorem}[section]
\newtheorem{proposition}[theorem]{Proposition}

\newtheorem{lemma}[theorem]{Lemma}

\newtheorem{observation}[theorem]{Observation}

\newtheorem{problem}[theorem]{Problem}

\newtheorem{defn}[theorem]{Definition}
\theoremstyle{definition}

\newcommand{\maj}{{\mathrm {maj}}}

\newcommand{\sign}{{\mathrm {sign}}}

\newcommand{\Val}{{\mathrm {Val}}}
\newcommand{\Rise}{{\mathrm {Rise}}}

\newcommand{\grFrob}{{\mathrm {grFrob}}}
\newcommand{\des}{{\mathrm {des}}}

\newcommand{\Conf}{{\mathrm {Conf}}}

\newcommand{\SYT}{{\mathrm {SYT}}}

\newcommand{\Frob}{{\mathrm {Frob}}}

\newcommand{\symm}{{\mathfrak{S}}}

\newcommand{\PP}{{\mathbb {P}}}
\newcommand{\CC}{{\mathbb {C}}}
\newcommand{\QQ}{{\mathbb {Q}}}
\newcommand{\ZZ}{{\mathbb {Z}}}

\newcommand{\one}{{\bm{1}}}

\newcommand{\xx}{{\mathbf {x}}}

\newcommand{\coFI}{{\mathsf {co}\text{-}\mathsf{FI}}}
\newcommand{\FI}{{\mathsf {FI}}}
\newcommand{\Vect}{{\mathsf {Vect}}}

\DeclareMathOperator{\im}{im}

\begin{document}

\title[Spanning subspace configurations and representation stability]
{Subspace configurations and representation stability}

\author{Brendan Pawlowski}
\address
{Department of Mathematics \newline \indent
University of Southern California \newline \indent
Los Angeles, CA, 90089, USA}
\email{bpawlows@usc.edu}

\author{Eric Ramos}
\address
{Department of Mathematics \newline \indent
University of Oregon \newline \indent
Eugene, OR, 97403, USA}
\email{eramos@uoregon.edu}

\author{Brendon Rhoades}
\address
{Department of Mathematics \newline \indent
University of California, San Diego \newline \indent
La Jolla, CA, 92093, USA}
\email{bprhoades@ucsd.edu}

\begin{abstract}
Let $V_1, V_2, V_3, \dots $ be a sequence of $\QQ$-vector spaces where $V_n$ carries
an action of $\symm_n$ for each $n$.  {\em Representation stability} and {\em multiplicity stability}
are two related notions of when the sequence $V_n$ has a limit.
An important source of stability phenomena arises in the case where $V_n$ is the $d^{th}$ homology group
(for fixed $d$) of the configuration space of $n$ distinct points in some fixed topological space $X$.
We replace these configuration spaces with the variety $X_{n,k}$ of 
{\em spanning configurations} of $n$-tuples $(\ell_1, \dots, \ell_n)$ of lines in
$\CC^k$ which satisfy $\ell_1 + \cdots + \ell_n = \CC^k$ as vector spaces.
We study stability phenomena for the homology groups $H_d(X_{n,k})$ as the parameter $(n,k)$ grows.
\end{abstract}

\maketitle

\section{Introduction and Main Results}
\label{Introduction}

Suppose that for each $n \geq 1$, we have a representation 
$V_n$ of the symmetric group $\symm_n$. 
\footnote{All representations considered in this paper will be finite-dimensional and over $\QQ$.}
What does it mean for the sequence $V_1, V_2, V_3, \dots $  to converge?
Two, a priori different, answers to this question  have arisen in the literature:
multiplicity stability and representation stability.
To describe these notions, we need some notation.

A {\em partition} of $n$ is a weakly decreasing sequence $\lambda = (\lambda_1, \lambda_2, \dots )$ of positive
integers with $\lambda_1 + \lambda_2 + \cdots = n$. We write $\lambda \vdash n$ to mean that $\lambda$
is a partition of $n$ and $|\lambda| = n$ for the sum of the parts of $\lambda$.
The {\em (English) Ferrers diagram} of $\lambda$ consists of $\lambda_i$ left-justified boxes in row $i$;
we identify partitions with their Ferrers diagrams.
The partition $(3,3,1) \vdash 7$ has the following Ferrers diagram:
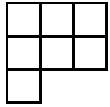
\begin{small}
\begin{center}
\begin{young}
 & & \cr
 & & \cr 
 \cr
\end{young}
\end{center}
\end{small}
There is a  correspondence between partitions of $n$ and irreducible representations
of $\symm_n$; given $\lambda \vdash n$, let $S^{\lambda}$ be the corresponding irreducible 
$\symm_n$-module.

If $\mu = (\mu_1, \mu_2, \dots )$ is a partition and $n \geq |\mu| + \mu_1$, the {\em padded partition}
is $\mu[n] \vdash n$ is given by $\mu[n] := (n - |\mu|, \mu_1, \mu_2, \dots )$.
Any partition $\lambda \vdash n$ may be expressed uniquely as $\lambda = \mu[n]$ for some partition $\mu$.
In fact, if $\lambda = (\lambda_1, \lambda_2, \lambda_3, \dots )$ then $\lambda = \mu[n]$
where $\mu = (\lambda_2, \lambda_3, \dots )$.

For any $n \geq 1$, the $\symm_n$-module $V_n$ decomposes into a direct sum of irreducibles.
There exist unique multiplicities $m_{\mu,n}$ so that 
\begin{equation}
V_n \cong \bigoplus_{\mu} m_{\mu,n}   S^{\mu[n]}
\end{equation}
where the direct sum is over all partitions $\mu$.

\begin{defn}
\label{multiplicity-stability-definition}
The sequence $(V_n)_{n \geq 1}$  is {\em uniformly multiplicity stable} if there exists 
$N$ such that for any partition $\mu$, we have $m_{\mu,n} = m_{\mu,n'}$ for all $n, n' \geq N$.
\end{defn}

Note that we set $m_{\mu,n} = 0$ whenever $n < |\mu| + \mu_1$. This then implies, for any multiplicity stable sequence of representations $(V_n)_{n \geq 1}$, that $m_{\mu,n} = 0$ for all but finitely many partitions $\mu$.

While multiplicity stability is a property of a sequence of $\symm_{n}$-representations, representation stability is concerned with such sequences which fit together in a consistent way. We embed $\symm_n \subseteq \symm_{n+1}$ by letting $\symm_n$ act on the first $n$ letters.
With respect to this embedding, any $\symm_{n+1}$-module is also an $\symm_n$-module.

\begin{defn}
\label{representation-stability-definition}
Let $(V_n)_{n \geq 1}$ be a sequence of $\symm_n$ representations, and for each $n \geq 1$ let $f_n:V_n \rightarrow V_{n+1}$ be a linear map. Then we say that $(V_n)_{n \geq 1}$ is {\em (uniformly) representation stable} with respect to the maps $(f_n)_{n \geq 1}$ if for $n \gg 0$
\begin{itemize}
\item the map $f_n$ is injective,
\item we have $f_n(w \cdot v) = w \cdot f_n(v)$ for all $w \in \symm_n$ and all $v \in V_n$,
\item the $\symm_{n+1}$ module generated by the image $f_n(V_n) \subseteq V_{n+1}$ is all of $V_{n+1}$, and
\item the transposition $(n+1,n+2) \in \symm_{n+2}$ acts trivially on the image of the composition $\im(V_n \stackrel{f_n}\rightarrow V_{n+1} \stackrel{f_{n+1}}\rightarrow V_{n+2})\subseteq V_{n+2}$.\label{consistent}
\end{itemize}
\end{defn}

An algebraic instance of representation stability arises from considering homogeneous polynomials.
Fix a polynomial degree $d$ and let $V_n := \QQ[x_1, \dots, x_n]_d$ be the space of polynomials in
$x_1, \dots, x_n$ which are homogeneous of degree $d$, with $\symm_n$ acting by subscript permutation.
If we let $f_n: V_n \hookrightarrow V_{n+1}$ be the inclusion, the sequence $(V_n)_{n \geq 1}$ is
representation stable.

Many geometric instances of representation stability arise from the homology groups of 
configuration spaces.
If $X$ is a topological space and $n \geq 1$, the $n^{th}$ {\em configuration space} of $X$ is 
the moduli space of $n$ distinct points in $X$:
\begin{equation}
\Conf_n X := \{ (x_1, \dots, x_n) \,:\, x_i \in X \text{ and $x_i \neq x_j$ for all $i \neq j$} \}
\end{equation}
The left of Figure~\ref{three-points} shows a point in $\Conf_3(X)$ where $X$ is the torus.
The set $\Conf_n X$ has the subspace topology inherited from the $n$-fold product 
$X \times \cdots \times X$.

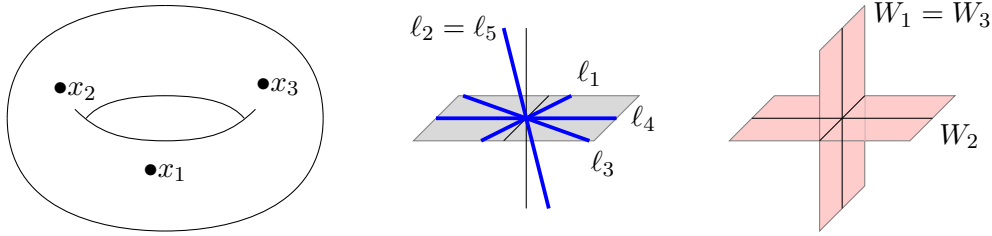
\begin{figure}

\begin{tikzpicture}[scale = 0.6]
\draw (-3.5,0) .. controls (-3.5,2) and (-1.5,2.5) .. (0,2.5);
\draw[xscale=-1] (-3.5,0) .. controls (-3.5,2) and (-1.5,2.5) .. (0,2.5);
\draw[rotate=180] (-3.5,0) .. controls (-3.5,2) and (-1.5,2.5) .. (0,2.5);
\draw[yscale=-1] (-3.5,0) .. controls (-3.5,2) and (-1.5,2.5) .. (0,2.5);

\draw (-2,.2) .. controls (-1.5,-0.3) and (-1,-0.5) .. (0,-.5) .. controls (1,-0.5) and (1.5,-0.3) .. (2,0.2);

\draw (-1.75,0) .. controls (-1.5,0.3) and (-1,0.5) .. (0,.5) .. controls (1,0.5) and (1.5,0.3) .. (1.75,0);

\draw [gray, fill=gray!30] (6.5,0.5) -- (10.5,0.5) -- (9.5,-0.5) -- (5.5,-0.5) -- (6.5,0.5);

\node at (0,-1.2) {$\bullet x_1$};

\node at (-2,0.6) {$\bullet x_2$};

\node at (2.5,0.7) {$\bullet x_3$};


\draw  (8,-2) -- (8,2);

\draw (6,0) -- (10,0);

\draw (8.5,0.5) -- (7.5,-0.5);

\draw [line width = 0.5mm, blue] (9,0.5) -- (7,-0.5);

\draw [line width = 0.5mm, blue] (6.6,0.5) -- (9.4,-0.5);

\draw [line width = 0.5mm, blue] (8.5,-2) -- (7.5,2);

\draw [line width = 0.5mm, blue] (6,0) -- (10,0);

\node at (6.4,2) {$\ell_2 = \ell_5$};

\node at (9.4,1) {$\ell_1$};

\node at (10.6,0) {$\ell_4$};

\node at (9.7,-1) {$\ell_3$};


\draw [gray, fill=red!20] (13.5,0.5) -- (17.5,0.5) -- (16.5,-0.5) -- (12.5,-0.5) -- (13.5,0.5);

\draw [gray, fill=red!20] (15.5,2.5) -- (14.5,1.5) -- (14.5,-2.5) -- (15.5,-1.5) -- (15.5,2.5);


\draw [gray] (15.5, -0.5) -- (14.5,-0.5);

\draw [red!20] (15.5, -0.5) -- (15.5,0);

\draw [red!20] (15.5, 0) -- (15.5,0.5);

\draw  (15,-2) -- (15,2);

\draw (13,0) -- (17,0);

\draw (15.5,0.5) -- (14.5,-0.5);


\node at (17.6,-0.4) {$W_2$};

\node at (17, 2.3) {$W_1 = W_3$};

\end{tikzpicture}

\caption{A point configuration, a line configuration, and a 2-plane configuration.}
\label{three-points}
\end{figure}

For $d \geq 0$, let $H_d(\Conf_n X )$ be the $d^{th}$ homology group of the $n^{th}$ configuration space 
of $X$. 
(Unless otherwise stated, we  use singular (co)homology with rational coefficients.)
The action of $\symm_n$ on $\Conf_n X$ by subscript permutation induces an $\symm_n$-action on the vector space
$H_d(\Conf_n X)$.
There are many results stating that if the space $X$ is `nice', the sequence
$(H_d(\Conf_n X))_{n \geq 1}$ is uniformly multiplicity stable \cite{Church, CEF}.

It is natural to ask whether there is any formal relationship between uniformly multiplicity stable sequences of representations, and representation stability with respect to a given collection of linear maps. Indeed, the following is one of the primary theorems of \cite{CEF}.

\begin{theorem}[\cite{CEF}]
Let $(V_n)_{n \geq 1}$ be a sequence of $\symm_n$-representations. Then $(V_n)_{n \geq 1}$ is uniformly multiplicity stable if and only if there exists some collection of linear maps $f_n:V_n \rightarrow V_{n+1}$ such that $(V_n)_{n \geq 1}$ is representation stable with respect to $(f_n)_{n \geq 1}$.
\end{theorem}

In practice, most natural examples of sequences $(V_n)_{n \geq 1}$ often come equipped with equally natural maps $(f_n:V_n \rightarrow V_{n+1})_{n \geq 1}$, as we saw in the above examples. The above theorem is therefore most useful in proving uniform multiplicity stability of the sequence, as one only needs to check the four relatively simple conditions of Definition \ref{representation-stability-definition}. This particular brand of argument also illustrates the following common theme, present throughout the theory:

\begin{quote}
{\bf Leitmotif.}
{\em Given a sequence $V_n$ of $\symm_n$-modules, it is often easier to prove that the sequence $V_n$
is uniformly multiplicity stable than it is to calculate the decomposition of each $V_n$ into irreducibles.}
\end{quote}

In this paper we consider stability in a family of spaces where the configuration
space condition of {\bf distinctness} is replaced by the matroidal condition of {\bf spanning}.

\begin{defn}
\label{x-space-definition}
Let $k \leq n$ be positive integers. The space $X_{n,k}$ consists of $n$-tuples of lines in the complex
vector space $\CC^k$ which span $\CC^k$:
\begin{equation*}
X_{n,k} := \{ (\ell_1, \dots, \ell_n) \,:\, \ell_i \subseteq \CC^k \text{ is a 
$1$-dimensional subspace and } \ell_1 + \cdots + \ell_n = \CC^k \}.
\end{equation*}
We also set $X_{n,k} := \varnothing$ if $k < n$ or $k < 0$.
\end{defn}

A point in $X_{5,3}$ is shown in the center of Figure~\ref{three-points}.
The space $X_{n,k}$ is homotopy equivalent to the variety $\mathcal{F \ell}_n$ 
of complete flags in $\CC^n$ when $k = n$.
For general $k \leq n$, the space $X_{n,k}$ serves as a kind of `flag variety' for the 
{\em Delta Conjecture} of symmetric function theory; see \cite{HRW, HRS, PR}.
The symmetric group $\symm_n$ acts on $X_{n,k}$ by line permutation:
$w.(\ell_1, \dots, \ell_n) := (\ell_{w(1)}, \dots, \ell_{w(n)})$.
This induces an action on the homology  $H_{\bullet}(X_{n,k})$.

The are two natural ways to grow the pair $(n,k)$ subject to the condition $k \leq n$:
\begin{center}
$(n,k) \leadsto (n+1,k)$ and $(n,k) \leadsto (n+1, k+1)$.
\end{center}
Both of these growth rules yield multiplicity stable homology representations.

\begin{theorem}
\label{multiplicity-theorem}
Fix a homological degree $d \geq 0$. 
\begin{enumerate}
\item For fixed $k \geq 0$, the sequence $(H_d(X_{n,k}))_{n \geq 1}$ is uniformly multiplicity stable.
\item For fixed $m \geq 0$, the sequence $(H_d(X_{n, n-m}))_{n \geq 1}$ is uniformly multiplicity stable.
\end{enumerate}
\end{theorem}

When $n-m < 0$, by convention $X_{n,n-m} = \varnothing$ so that $H_d(X_{n,n-m}) = H_d(\varnothing) = 0$.
Fixing $d$ in Theorem~\ref{multiplicity-theorem} is necessary. In Part (2) when $m = 0$ we have
$H_{\bullet}(X_{n,n}) \cong H_{\bullet}(\mathcal{F \ell}_n) \cong \QQ[\symm_n]$, the regular representation
of $\symm_n$.
The sequence $(\QQ[\symm_n])_{n \geq 1}$ of regular representations is not multiplicity stable because
$\QQ[\symm_n]$ has a copy of the sign representation $S^{(1^n)}$ for each $n$.
Theorem~\ref{multiplicity-theorem} will be proven in Section~\ref{Proofs}.

We next turn our attention to constructing maps between the $X_{n,k}$, so as to prove representation stability statements. Let $\one_k \subseteq \CC^k$ be the 
line of constant vectors in $\CC^k$ and let $\iota: \CC^k \hookrightarrow \CC^{k+1}$ be 
$\iota: (x_1, \dots, x_k) \mapsto (x_1, \dots, x_k, 0)$.
We have natural injections
\begin{equation*}
i: X_{n,k} \hookrightarrow X_{n+1,k} \quad \text{and} \quad j: X_{n,k} \hookrightarrow X_{n+1,k+1}
\end{equation*}
given by
\begin{equation*}
i: (\ell_1, \dots, \ell_n) \mapsto (\ell_1, \dots, \ell_n, \one_k) \quad \text{and} \quad
j: (\ell_1, \dots, \ell_n) \mapsto (\iota(\ell_1), \dots, \iota(\ell_n), \one_{k+1}).
\end{equation*}
For $d \geq 0$ fixed, let $f_n: H_d(X_{n,k}) \rightarrow H_d(X_{n+1,k})$ be the homology map 
induced by $i$ and let $g_n: H_d(X_{n,k}) \rightarrow H_d(X_{n+1,k+1})$ be the homology map
induced by $j$.

\begin{theorem}
\label{stability-theorem}
Fix a homological degree $d \geq 0$. 
\begin{enumerate}
\item For fixed $k \geq 0$, the sequence $(H_d(X_{n,k}))_{n \geq 1}$ is representation stable
with respect to the maps $f_n$.
\item For fixed $m \geq 0$, the sequence $(H_d(X_{n, n-m}))_{n \geq 1}$ is representation stable
with respect to the maps $g_n$.
\end{enumerate}
\end{theorem}

Theorem~\ref{stability-theorem} implies Theorem~\ref{multiplicity-theorem}.
Both of these results will be proven in Section~\ref{Proofs}.

The irreducible decompositions of the homology representations $H_d(X_{n,k})$ appearing in Theorems~\ref{multiplicity-theorem} and
\ref{stability-theorem} have been described in terms of statistics on standard Young tableaux,
but this calculation 
relies  on a significant amount of combinatorics, algebra, and geometry
\cite{HRS, PR, WMultiset}.  
The  leitmotif  suggests that 
Theorems~\ref{multiplicity-theorem} and \ref{stability-theorem} should be provable
without using the full machinery of \cite{HRS, PR, WMultiset}.

We will give three different proofs of our stability results, each of a different flavor (combinatorial,
algebraic, and geometric)
and relying on a different size of `black box' 
on the structure of the groups $H_d(X_{n,k})$.
\begin{center}
\begin{tabular}{c | c  }
Proof Method & Black Box Size  \\ \hline 
Combinatorial & Large (see Theorem~\ref{graded-frobenius})  \\  \hline
Algebraic & Medium (see Theorem~\ref{cohomology-presentation}) \\ \hline
Geometric & Small (see Theorem~\ref{affine-paving})
\end{tabular}
\end{center}
In Theorem~\ref{higher-dimensional-stability} we extend our algebraic approach to obtain stability results
for homology groups of moduli spaces of spanning configurations of higher dimensional subspaces;
see the right of Figure~\ref{three-points}.
These latter homology groups were presented in \cite{RhoadesSpanning} in terms of a polynomial ring quotient,
but their decomposition into irreducible $\symm_n$-modules is unknown; this is a further illustration
of the  leitmotif.

The remainder of the paper is organized as follows. In {\bf Section~\ref{Black}} we 
discuss the black boxes we will use to prove Theorems~\ref{multiplicity-theorem} and
\ref{stability-theorem}.
In {\bf Section~\ref{Proofs}} we give our three proofs of these theorems.
In {\bf Section~\ref{Other}} we extend Theorem~\ref{stability-theorem} in two directions:
to higher-dimensional subspaces (Theorem~\ref{higher-dimensional-stability}) 
and to diagonal $\symm_n$-actions on rings with commuting and anticommuting generators 
(Propositions~\ref{s-is-stable} and \ref{r-is-stable}).

\section{Black Boxes}
\label{Black}

Let $Z$ be a smooth complex algebraic variety. An {\em affine paving} of $Z$ is a chain
\begin{equation*}
\varnothing = Z_0 \subset Z_1  \subset \cdots \subset Z_m = Z 
\end{equation*}
of subvarieties of $Z$ such
that each difference $Z_i - Z_{i-1}$ is isomorphic to a disjoint union of affine spaces.
If $\varnothing = Z_0 \subset Z_1  \subset \cdots \subset Z_m = Z$ is an affine paving and $0 \leq j \leq m$,
let $U := Z - Z_j$.
The inclusion $i: U \hookrightarrow Z$ induces maps on homology and cohomology:
\begin{equation}
i_*: H_{\bullet}(U) \rightarrow H_{\bullet}(Z) \quad \text{and} \quad
i^*: H^{\bullet}(Z) \rightarrow H^{\bullet}(U).
\end{equation}
When $U$ arises from an affine paving of $Z$ as above, the map $i_*$ is injective and the map
$i^*$ is surjective.

The standard affine paving of complex projective space $\PP^{k-1}$ given in coordinates by
\begin{equation*} \varnothing \subset
[ \star : 0 : \cdots : 0 : 0 ] \subset  [\star : \star : \cdots : 0 : 0 ] \subset \cdots \subset 
[\star : \star : \cdots : \star : 0] \subset [\star : \star : \cdots : \star : \star ] = \PP^{k-1}.
\end{equation*}
Taking the $n$-fold product of this paving with itself yields a product paving of $(\PP^{k-1})^n$, but this 
paving interacts poorly with the inclusion $X_{n,k} \subseteq (\PP^{k-1})^n$.
A nonstandard affine paving of $(\PP^{k-1})^n$ was crucial to the presentation of 
the cohomology of $X_{n,k}$ in \cite{PR}.

\begin{theorem}
\label{affine-paving} (\cite{PR})
{\em ``Small Black Box"} Let $k \leq n$ be positive integers.
There exists an affine paving $Z_0 \subset Z_1 \subset \cdots \subset Z_m$ of the $n$-fold 
product $(\PP^{k-1})^n$ such that $X_{n,k} = (\PP^{k-1})^n - Z_i$ for some $i$.
Consequently, if $i: X_{n,k} \hookrightarrow (\PP^{k-1})^n$ is inclusion, the induced maps 
on homology and cohomology
\begin{equation}
i_*: H_{\bullet}(X_{n,k}) \hookrightarrow H_{\bullet}((\PP^{k-1})^n) \quad \text{and} \quad
i^*: H^{\bullet}((\PP^{k-1})^n) \twoheadrightarrow H^{\bullet}(X_{n,k})
\end{equation}
are injective and surjective, respectively.
\end{theorem}

The following lemma allows us to transfer between homology
and cohomology at will.

\begin{lemma}
\label{homology-cohomology-transfer}
Let $X$ be $X_{n,k}$ or $(\PP^{k-1})^n$ and $d \geq 0$.
We have
a natural $\symm_n$-equivariant
isomorphism $\mathrm{Hom}_{\QQ}(H^i(X), \QQ) \cong H_i(X)$.  
\end{lemma}

\begin{proof}
Recall that we are considering rational cohomology.
The Universal Coefficient Theorem gives a short exact sequence
\begin{equation}
0 \rightarrow \mathrm{Ext}^1_{\QQ}(H_{d-1}(X), \QQ) \rightarrow H^d(X) \xrightarrow{h}
\mathrm{Hom}_{\QQ}(H_d(X), \QQ) \rightarrow 0
\end{equation}
where $h$ is the map $h([f])[x] = f(x)$ for any $[f] \in H^d(X)$ and any $[x] \in H_d(X)$.
Since $\QQ$ is a field, the $\mathrm{Ext}$-group vanishes and $h$ is a manifestly $\symm_n$-equivariant
isomorphism.  When $H_d(X)$ is a finite-dimensional $\QQ$-vector space, we may dualize to get the desired
isomorphism.
\end{proof}

For $1 \leq d \leq n$, let $e_d = e_d(x_1, \dots, x_n) \in \QQ[x_1, \dots, x_n]$ be the degree $d$ elementary
symmetric polynomial
\begin{equation*}
e_d := \sum_{1 \leq i_1 < \cdots < i_d \leq n} x_{i_1} \cdots x_{i_d}.
\end{equation*}
For positive integers $k \leq n$, let $I_{n,k} \subseteq \QQ[x_1, \dots, x_n]$ be the ideal
\begin{equation}
I_{n,k} := \langle x_1^k, x_2^k, \dots, x_n^k, e_n, e_{n-1}, \dots, e_{n-k+1} \rangle \subseteq \QQ[x_1, \dots, x_n]
\end{equation}
and let $R_{n,k} := \QQ[x_1, \dots, x_n]/I_{n,k}$ be the corresponding quotient ring.
Since the ideal $I_{n,k}$ is homogeneous and $\symm_n$-stable, the quotient $R_{n,k}$ is a graded ring
and a graded $\symm_n$-module.
The following `black box' result follows quickly from \cite{PR}.

\begin{theorem}
\label{cohomology-presentation}  {\em ``Medium Black Box"}
Let $k \leq n$ be positive integers. We have a presentation of the rational cohomology ring
$H^{\bullet}(X_{n,k}) = R_{n,k}$ where the variable $x_i$ represents the Chern class $c_1(\ell_i^*) \in H^2(X_{n,k})$ 
of the 
tautological line bundle $\ell_i^* \twoheadrightarrow X_{n,k}$.
\end{theorem}

\begin{proof}
In \cite{PR} the corresponding statement is proven over $\ZZ$: the integral cohomology
$H^{\bullet}(X_{n,k}; \ZZ)$ is presented as the quotient 
$R_{n,k}^{\ZZ} := \ZZ[x_1, \dots, x_n]/\langle x_1^k, \dots, x_n^k, e_n, \dots, e_{n-k+1} \rangle$ with the variable
$x_i$ representing the Chern class $c_1(\ell_i^*) \in H^2(X_{n,k}; \ZZ)$.  
In particular, the integral cohomology of $X_{n,k}$ 
is concentrated in even dimensions, so the Universal Coefficient Theorem implies 
\begin{equation}
H^{\bullet}(X_{n,k}; \QQ) = \QQ \otimes_{\ZZ} H^{\bullet}(X_{n,k}; \ZZ) = \QQ \otimes_{\ZZ} R_{n,k}^{\ZZ} =
R_{n,k}
\end{equation}
as desired.
\end{proof}
Given Theorem~\ref{cohomology-presentation} and Lemma~\ref{homology-cohomology-transfer}, describing the 
$\symm_n$-isomorphism type of $H_d(X_{n,k})$ is equivalent to decomposing the degree $d$ piece of $R_{n,k}$
into irreducibles.
The graded isomorphism type of $R_{n,k}$ was calculated by Haglund, Rhoades, and Shimozono \cite{HRS}.
To describe this isomorphism type, we need some notation.

Let $\Lambda = \bigoplus_{n \geq 0} \Lambda_n$ be the ring of symmetric functions over the ground field
$\QQ(q,t)$.  The  space $\Lambda_n$ of homogeneous degree $n$ symmetric functions has a basis
of {\em Schur functions} $\{ s_{\lambda} \,:\, \lambda \vdash n \}$.

If $V$ is a finite-dimensional $\symm_n$-module with irreducible decomposition 
$V \cong \bigoplus_{\lambda \vdash n} m_{\lambda} S^{\lambda}$, the {\em Frobenius image} of $V$ 
is the symmetric function $\Frob(V) := \sum_{\lambda \vdash n} m_{\lambda} s_{\lambda} \in \Lambda_n$.
More generally, if $V = \bigoplus_{i \geq 0} V_i$ is a graded $\symm_n$-module with each
$V_i$ finite-dimensional, the {\em graded Frobenius image} is 
$\grFrob(V; q) := \sum_{i \geq 0} \Frob(V_i) \cdot q^i$.

Given $\lambda \vdash n$,
a {\em standard Young tableau} of shape $\lambda$ is a filling of the Ferrers diagram of $\lambda$
with $1, 2, \dots, n$ which is increasing down columns and across rows.
We let $\SYT(\lambda)$ be the family of standard Young tableaux of shape $\lambda$
and let $\SYT(n)$ be the family of all standard Young tableaux with $n$ boxes.
If $T \in \SYT(n)$, we write $\mathrm{shape}(T) \vdash n$ for the underlying partition of $T$.

Let $T$ be a standard Young tableau with $n$ boxes. 
An index $1 \leq i \leq n-1$ is called a {\em descent} of $T$
if $i$ appears in a row strictly above the row containing $i+1$ in $T$.
Let $\des(T)$ be the number of descents of $T$ and $\maj(T)$ be the sum of the descents of $T$.
For example, the tableau $T$  shown below has descents 1, 3, and 6 so that $\des(T) = 3$ and 
$\maj(T) = 1+3+6 = 10$.
\begin{small}
\begin{center}
\begin{young}
1 & 3 & 5 & 6 \cr
2 & 4 & 8 \cr
7
\end{young}
\end{center}
\end{small}
We  recall the $q$-analogues of numbers, factorials, and binomial coefficients:
\begin{equation}
[n]_q := 1 + q + \cdots + q^{n-1} \quad
[n]!_q := [n]_q [n-1]_q \cdots [1]_q \quad
{n \brack k}_q := \frac{[n]!_q}{[k]!_q \cdot [n-k]!_q}.
\end{equation}
The graded isomorphism type of $H^{\bullet}(X_{n,k})$ may be described in terms of standard tableaux as follows.

\begin{theorem}
\label{graded-frobenius} (\cite{HRS, PR})
{\em ``Large Black Box"}
Let $k \leq n$ be positive integers. We have
\begin{equation}
\grFrob(H^{\bullet}(X_{n,k}); \sqrt{q}) = \sum_{T \in \SYT(n)} q^{\maj(T)} {n - \des(T) - 1 \brack n-k}_q 
s_{\mathrm{shape}(T)}.
\end{equation}
\end{theorem}

The presence of $\sqrt{q}$ rather than $q$ on the left-hand-side
 of Theorem~\ref{graded-frobenius} stems from the generators
$x_i \leftrightarrow c_1(\ell_i^*) \in H^2(X_{n,k})$ of the polynomial ring $\QQ[x_1, \dots, x_n]$
representing degree 2 elements in cohomology.
Theorem~\ref{graded-frobenius} is our largest black box; it depends on the cohomology
presentation of Theorem~\ref{cohomology-presentation} as well as additional algebraic and combinatorial 
arguments \cite{HRS} (such as a generalization of the RSK correspondence 
from words to ordered multiset partitions due to Wilson \cite{WMultiset}).

\section{Proofs of Theorems~\ref{multiplicity-theorem} and \ref{stability-theorem}: Combinatorics, Algebra, and 
Geometry}
\label{Proofs}

\subsection{Combinatorics}
We prove the uniform multiplicity stability asserted in Theorem~\ref{multiplicity-theorem}
using the combinatorics of tableau statistics. 
We will assume our Large Black Box Theorem~\ref{graded-frobenius} as well as the standard combinatorial
interpretation of the $q$-binomial coefficient
\begin{equation}
{n \brack k}_q = \sum_{\lambda} q^{|\lambda|}
\end{equation}
where the sum is over all partitions $\lambda$ whose Ferrers diagrams fit inside a $k \times (n-k)$ rectangle.
We will also use  the following observation.

\begin{observation}
\label{major-index-observation}
Let $\mu$ be a partition and $n \geq |\mu| + \mu_1$.  Any standard tableau $T$ of shape $\mu[n]$
has $\maj(T) \geq |\mu|$.
\end{observation}

We proceed to give a combinatorial proof of Theorem~\ref{multiplicity-theorem}.

\begin{proof} (of Theorem~\ref{multiplicity-theorem})
(1) If $d$ is odd, the sequence $H_d(X_{n,k})$ of homology representations in question is the zero 
sequence and there is nothing to show, so assume $d = 2s$ is even.

By Lemma~\ref{homology-cohomology-transfer} and Theorem~\ref{graded-frobenius}, for $\lambda \vdash n$
the multiplicity of $S^{\lambda}$ in $H_d(X_{n,k})$ is the size of the set
\begin{equation}
\label{first-pair-set} A(\lambda,k) :=
\left\{ (T, \nu) \,:\, \begin{matrix}
\text{$T$ is a standard tableau of shape $\lambda$,} \\
\text{$\nu$ is a partition inside a $(k - \des(T) - 1) \times (n-k)$ rectangle,} \\
\text{and $\maj(T) + |\nu| = s$.}
\end{matrix} \right\}.
\end{equation}
We want to understand $|A(\lambda,k)|$ in the limit $n \rightarrow \infty$ with $k$ and $s$ fixed.  

If $\lambda = (\lambda_1, \lambda_2, \lambda_3, \dots )$ write $\mu = (\lambda_2, \lambda_3, \dots )$
so that $\lambda = \mu[n]$.  We have a map
\begin{equation}
\varphi: \{ T \in \SYT(\mu[n]) \,:\, \maj(T) \leq s \} \rightarrow \{U \in \SYT(\mu[n+1]) \,:\, \maj(U) \leq s \}
\end{equation}
where $\varphi(T)$ is obtained from $T$ by adding a box labeled $n+1$ to the first row.
Since adding $n+1$ to the first row does not introduce a descent, we have
$\des(\varphi(T)) = \des(T)$ and $\maj(\varphi(T)) = \maj(T)$ for any $T$ in the domain of $\varphi$.
If $n > 2s$, the map $\varphi$ is a bijection.

Suppose $0 \leq t \leq s$ and $T \in \SYT(\mu[n])$ satisfies $\maj(T) = t$.  We necessarily have $\des(T) \leq t$.
How many partitions $\nu \vdash (s-t)$ fit inside a $(k - \des(T) - 1) \times (n-k)$ rectangle as $n \rightarrow \infty$?
The answer is best understood with a picture: 
in the following diagram, the shaded partition $\nu$ has $s-t$ boxes
and the height $k - \des(T) - 1$ of the rectangle
 is fixed. However, the width $n-k$ of the rectangle is growing linearly with $n$.

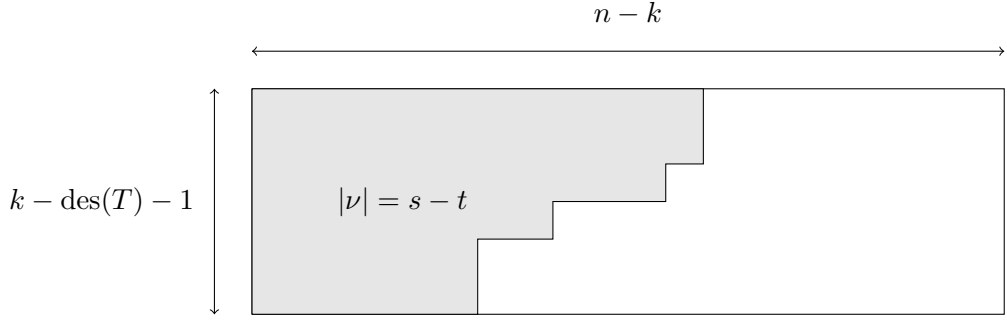
\begin{center}
\begin{tikzpicture}

\draw (0,0) -- (10,0) -- (10,3) -- (0,3) -- (0,0);

\draw [<->] (0,3.5) -- (10,3.5);

\draw [<->] (-0.5,0) -- (-0.5,3);

\node at (-2,1.5) {$k - \des(T) - 1$};

\node at (5,4) {$n-k$};

\draw [fill=gray!20] (0,3) -- (6,3) -- (6,2) -- (5.5,2) -- (5.5,1.5) -- (4,1.5) -- (4,1) -- (3,1) -- (3,0) -- (0,0) -- (0,3);

\node at (2,1.5) {$|\nu| = s-t$};

\end{tikzpicture}
\end{center}
If we assume $n > s + k$, the right vertical boundary provides no obstruction to the 
Ferrers diagram of $\nu \vdash s-t$ and the number of such partitions $\nu$ is 
\begin{equation}
p(s-t, \leq k - \des(T) - 1) := 
\text{the number of partitions of $s-t$ with $\leq k - \des(T) - 1$ parts,}
\end{equation}
a quantity which is independent of $n$.
By the previous paragraph we have
\begin{equation}
|A(\mu[n],k)| = |A(\mu[n+1],k)| \quad \text{for all $n > \max(2s, s+k)$},
\end{equation}
and uniform multiplicity stability follows.

(2) Let $\mu$ be a partition. We show
$|A(\mu[n],n-m)| = |A(\mu[n+1],n-m+1)|$ for $n$ sufficiently large.
If $|\mu| > s$ Observation~\ref{major-index-observation} implies
$A(\mu[n],n-m) = A(\mu[n+1],n-m+1) = \varnothing$.
We therefore assume $|\mu| \leq s$.

Let $0 \leq t \leq s$ and let $T \in \SYT(\mu[n])$ satisfy $\maj(T) = t$, so that $\des(T) \leq t$.
How many partitions $\nu \vdash (s-t)$ fit inside a $(n-m-\des(T)-1) \times m$ rectangle?
If $n > s + m + \des(T) + 1$, reasoning analogous to the proof of (1)
(but with a rectangle of fixed width and growing height) shows that 
the number of such partitions $\nu$ is 
\begin{equation}
p(s-t, \leq m) :=
\text{the number of partitions of $s-t$ with all parts $\leq m$,} 
\end{equation}
a quantity independent of $n$.  Applying the bijection $\varphi$, we conclude that
\begin{equation}
|A(\mu[n],n-m)| = |A(\mu[n+1],n-m+1)| \quad \text{for all $n > \max(2s+m+1, |\mu| + \mu_1)$}.
\end{equation}
Since $\mu_1 \leq |\mu| \leq s$, uniform multiplicity stability is proven.
\end{proof}

\subsection{Algebra}
We give an algebraic proof of Theorem~\ref{stability-theorem} using $\coFI$-modules.
We assume the Medium Black Box Theorem~\ref{cohomology-presentation} which presents 
the cohomology ring $H^{\bullet}(X_{n,k}) = R_{n,k}$ but {\bf not} the Large Black Box
Theorem~\ref{graded-frobenius} which describes the graded $\symm_n$-isomorphism 
type of $R_{n,k}$.

Let $\FI$ be the category whose objects are the finite sets 
$[n] := \{1, 2, \dots, n\}$ for $n \geq 0 $ (where we take $[0] := \varnothing$)
and whose morphisms are injective functions $f: [n] \rightarrow [m]$.
Let $\Vect$ be the category of $\QQ$-vector spaces with morphisms given by linear maps.

An {\em $\FI$-module} is a covariant functor $V: \FI \rightarrow \Vect$.  
We write $V(n)$ instead of $V([n])$ for the vector space corresponding to $[n]$.
More explicitly, an 
$\FI$-module consists of 
\begin{itemize}
\item a $\QQ$-vector space $V(n)$ for each $n \geq 0$, and
\item a $\QQ$-linear map $V(f): V(n) \rightarrow V(m)$ for each injection $f: [n] \rightarrow [m]$
\end{itemize}
such that $V(f \circ g) = V(f) \circ V(g)$ for any two injections $f: [n] \rightarrow [m]$ and $g: [m] \rightarrow [r]$
and $V( \mathrm{id}_{[n]}) = \mathrm{id}_{V(n)}$.
Submodules and quotients of $\FI$-modules are defined in the natural way.
Observe that if $V$ is an $\FI$-module, then $V(n)$ is naturally an $\symm_n$-module for each $n$.

The crucial example of an $\FI$-module is the assignment $V(n) := \QQ[x_1, \dots, x_n]$ 
which sends $[n]$ to the polynomial ring in $x_1, \dots, x_n$.
Given an injection $f: [n] \rightarrow [m]$,
the functor $V$ associates the linear map $V(f): \QQ[x_1, \dots, x_n] \rightarrow \QQ[x_1, \dots, x_m]$
induced by $x_i \mapsto x_{f(i)}$.
For any fixed $d \geq 0$, the $\FI$-module $V$ has a submodule $V_d: [n] \rightarrow \QQ[x_1, \dots, x_n]_d$
which sends $n$ to the space of polynomials in $x_1, \dots, x_n$ which are homogeneous of degree $d$.

An $\FI$-module $V$ is {\em finitely generated} if there is a finite subset $S$ of
the disjoint union $\bigsqcup_{n \geq 0} V(n)$
such that no proper $\FI$-submodule $W$ of $V$ satisfies $S \subseteq \bigsqcup_{n \geq 0} W(n)$.
Although the $\FI$-module $V$ of the above paragraph is not finitely generated, its $\FI$-submodules
$V_d$ are finitely generated for every fixed $d \geq 0$.
In fact, the $\FI$-module $V_d$ is generated by the set of monomials
\begin{equation}
S = \{ x_1^{\lambda_1} \cdots x_d^{\lambda_d} \,:\, \lambda \vdash d \} 
\end{equation}
in $x_1, \dots, x_d$ whose exponent sequences are partitions 
$\lambda = (\lambda_1 \geq \cdots \geq \lambda_d \geq 0)$ of $d$.
If $\lambda$ has $k \leq d$ positive parts, we consider the monomial $x_1^{\lambda_1} \cdots x_k^{\lambda_k}$
to lie in the space $V(k)$.

Importantly, finite generation of $\FI$-modules and representation stability are essentially equivalent.

\begin{theorem}[\cite{CEF}]\label{FIRepStable}
Let $\iota_n:[n] \hookrightarrow [n+1]$ denote the standard injection, and let $V$ be a finitely generated $\FI$-module. Then the sequence $(V(n))_{n \geq 1}$ is representation stable with respect to the maps $V(\iota_n)$. Conversely, every representation stable sequence of representations arises in this way.
\end{theorem}

Theorem \ref{FIRepStable} is perhaps most useful when one can prove finite generation of a given $\FI$-module using the following "Noetherian Property."

\begin{theorem}
\label{fi-is-noetherian}  (\cite{Snowden}; see also \cite{CEF})
Any submodule of a quotient of a finitely generated $\FI$-module is finitely generated.
\end{theorem}

We will see that often times one can embed a complicated $\FI$-module into one which is significantly simpler and clearly finitely generated. In this circumstance, the above foundational theorems allow us to conclude uniform multiplicity stability of the original sequence almost for free. This will be our strategy going forward.

Let $V$ be the $\FI$-module $V(n) = \QQ[x_1, \dots, x_n]$ considered above. 
For fixed $k \geq 1$, the assignment 
\begin{equation}
[n] \mapsto \begin{cases}
I_{n,k} & \text{if $n \geq k$,} \\
\QQ[x_1, \dots, x_n] & \text{if $n < k$.} 
\end{cases}
\end{equation}
does {\em not} define an $\FI$-submodule of $V$.  Indeed, consider the inclusion
$i: [n] \hookrightarrow [n+1]$ for some $n \geq k$.  We have 
$e_n(x_1, \dots, x_n) = x_1 x_2 \cdots x_n \in I_{n,k}$ but
$V(i) : x_1 x_2 \dots x_n \mapsto x_1 x_2 \cdots x_n \notin I_{n+1,k}$.

To fit the
 ideals $I_{n,k}$ and rings $R_{n,k}$
  into the $\FI$ framework, we need the dual notion of a {\em $\coFI$-module}.
Let $\coFI$ be the {\em opposite category} to $\FI$ obtained by arrow reversal.
A {\em $\coFI$-module} is a covariant functor $\coFI \rightarrow \Vect$.  
If $V$ is an $\FI$-module, its dual $V^*$ is naturally a $\coFI$-module.
We apply $\coFI$-modules to prove Theorem~\ref{stability-theorem} as follows.

\begin{proof} (of Theorem~\ref{stability-theorem})
If $d$ is odd we have $H_d(X_{n,k}) = 0$ for all $k \leq n$ and there is nothing to prove, so assume $d = 2s$
is even for the remainder of the proof.

We are assuming the Medium Black Box Theorem~\ref{cohomology-presentation} which gives a presentation
\begin{equation}
H^{\bullet}(X_{n,k}) = \QQ[x_1, \dots, x_n]/I_{n,k}.
\end{equation}
By Lemma~\ref{homology-cohomology-transfer} we have an identification of vector spaces
\begin{equation}
H_d(X_{n,k}) = \QQ[x_1, \dots, x_n]_s/(I_{n,k} \cap \QQ[x_1, \dots, x_n]_s)
\end{equation}
for any $k \leq n$. 
If $k > n$ we have $H_d(X_{n,k}) = H_d(\varnothing) = 0$.

We have a $\coFI$-module $W_s(n) := \QQ[x_1, \dots, x_n]_s$ which assigns to an injection
$f: [n] \rightarrow [n']$ the linear map $W_s(f): W_s(n') \rightarrow W_s(n)$ induced by
\begin{equation*}
x_j \mapsto 
\begin{cases}
x_{i} & \text{if $j = f(i)$}  \\
0 & \text{if $j \notin \mathrm{Image}(f)$}
\end{cases}
\end{equation*}
for $1 \leq j \leq n'$.
Said differently, the $\coFI$-module $W_s$ is the dual of the $\FI$-module $V_s$ defined above.

Fix $k$ and $m$ as in the statement of Theorem~\ref{stability-theorem}.
We define two submodules $P_s$ and $Q_s$ of the $\coFI$-module $W_s$ as follows.
\begin{equation}
P_s(n) := \begin{cases}
I_{n,k} \cap W_s(n) & n \geq k \\
W_s(n) & n < k
\end{cases}  \quad \quad
Q_s(n) := \begin{cases}
I_{n,n-m} \cap W_s(n) & n \geq m \\
W_s(n) & n < m
\end{cases}
\end{equation}
The fact that $P_s$ and $Q_s$ are in fact $\coFI$-submodules of $W_s$ amounts to checking that 
for any injection $f: [n] \rightarrow [n']$ we have the containments 
\begin{center}
$W_s(f)(P_s(n')) \subseteq P_s(n)$
and $W_s(f)(Q_s(n')) \subseteq Q_s(n)$.
\end{center}
Given the generating set of $I_{n,k}$, this reduces to the observation that 
\begin{center}
$W_s(f)(e_d(x_1, \dots, x_{n'})) = e_d(x_1, \dots, x_n)$ for any $1 \leq d \leq n'$ and any
injection $f: [n] \rightarrow [n']$
\end{center}
where we interpret $e_d(x_1, \dots, x_n) = 0$ if $n < d$.
Theorem~\ref{stability-theorem} follows from Theorem~\ref{fi-is-noetherian}
and the work of Church, Ellenberg, and Farb \cite[Thm. 1.13]{CEF}.
\end{proof}

\subsection{Geometry}
In this subsection we give another proof of Theorem~\ref{stability-theorem} which assumes 
only the Small Black Box Theorem~\ref{affine-paving} that $X_{n,k}$ is a 
union of open cells in
a (nonstandard) affine 
paving of $(\PP^{k-1})^n$.
This does not rely on the presentation $H^{\bullet}(X_{n,k}) = R_{n,k}$ (or on the description of the graded 
$\symm_n$-structure of $R_{n,k}$).
Like the previous proof,
we use the category $\FI$.  However,
 this proof is geometric in nature and works directly
with the spaces $X_{n,k}$.

\begin{proof} (of Theorem~\ref{stability-theorem})
We begin with statement (1).
Let $f: [n] \rightarrow [p]$ be an injective map.  We define a map
$\iota_f: (\PP^{k-1})^n \hookrightarrow (\PP^{k-1})^p$
as follows.  Given an $n$-tuple
$(\ell_1, \dots, \ell_n) \in (\PP^{k-1})^n$ we set $\iota_f(\ell_1, \dots, \ell_n) := (\ell'_1, \dots, \ell'_p)$ where
\begin{equation}
\ell'_j = \begin{cases}
\ell_i & \text{if $f(i) = j$} \\
\one_k & \text{if $j \notin \mathrm{Image}(f)$.}
\end{cases}
\end{equation}
(Recall that $\one_k \in \PP^{k-1}$ is the line of constant vectors in $\CC^k$.)
The map $\iota_f$ preserves the spanning property of line configurations
in the sense that 
$\iota_f(X_{n,k}) \subseteq X_{p,k}$.
Abusing notation, we write $\iota_f$ for the restricted map $X_{n,k} \rightarrow X_{p,k}$.

For any fixed degree $d$, we have two $\FI$-module structures coming from the homology groups 
$H_d((\PP^{k-1})^n)$ and the homology groups $H_d(X_{n,k})$.  If $f: [n] \rightarrow [p]$ is any injection
we have a commutative square

\begin{center}
\begin{tikzpicture}[every node/.style={midway}]
\matrix[column sep={10em,between origins},
        row sep={4em}] at (0,0)
{ \node(A)   {$H_d(X_{n,k})$}  ; & \node(B) {$H_d((\PP^{k-1})^n)$}; \\
  \node(C) {$H_d(X_{p,k})$}; &\node(D){$H_d((\PP^{k-1})^p)$};                  \\};
\draw[->] (B) -- (D) node[anchor=west]  {$(\iota_f)_d$};
\draw[->] (A) -- (C) node[anchor=east]  {$(\iota_f)_d$};
\draw[->] (A)   -- (B) node[anchor=south] {};
\draw[->] (C)   -- (D) node[anchor=south] {};
\end{tikzpicture}
\end{center}
Here $(\iota_f)_d$ is the induced map from $\iota_f$
on  $d^{th}$ homology groups and the horizontal maps are induced by the inclusions
$X_{n,k} \subseteq (\PP^{k-1})^n$ and $X_{p,k} \subseteq (\PP^{k-1})^p$.
By our Small Black Box Theorem~\ref{affine-paving}, the horizontal maps are injections, so that 
$H_d(X_{n,k})$ is a submodule of $H_d((\PP^{k-1})^n)$.  Since $H_d((\PP^{k-1})^n)$ is finitely generated,
Theorem~\ref{fi-is-noetherian} implies that $H_d(X_{n,k})$ is also finitely generated.
The work of Chuch, Ellenberg, and Farb \cite[Thm. 1.13]{CEF} finishes the proof 
of Theorem~\ref{stability-theorem} (1).

We turn to the proof of Theorem~\ref{stability-theorem} (2). This follows the same basic strategy as in
Theorem~\ref{stability-theorem} (1) but there are additional complications involving the gap between
maps on spaces and their induced maps on homology.

For any injection $f: [n] \rightarrow [p]$, 
write $[p] - \mathrm{Image}(f) := \{ i_1 < i_2 < \cdots < i_{p-n} \}$ for the elements 
of $[p]$ which do not lie in the image of $f$.
Define 
$\nu_f: (\PP^{n-m-1})^n \rightarrow (\PP^{p-m-1})^p$ by 
$\nu_f(\ell_1, \dots, \ell_n) = (\ell''_1, \dots, \ell''_p)$ where
\begin{equation}
\ell''_j = \begin{cases}
\ell_i & \text{if $f(i) = j$} \\
\mathrm{span} \{ e_{t + n - m} \} & \text{if $j = i_t \notin \mathrm{Image}(f)$.}
\end{cases}
\end{equation}
Here we view $\ell_i$ as a line in $\CC^{p-m}$ via the embedding $\CC^{n-m} \subseteq \CC^{p-m}$ into the
first $p-m$ coordinates and
$\mathrm{span} \{ e_{t + n - m} \}$ is the line spanned by the $(t+n-m)^{th}$ standard coordinate vector in
$\CC^{p-m}$.
For example, if $f: [3] \rightarrow [5]$ is the injection $f(1) = 4, f(2) = 1, f(3) = 3$ then for $n-m = 2$ we have
\begin{equation*} \nu_f:
\begin{pmatrix}
3 & 1 & 2 \\
-1 & 1 & 1 
\end{pmatrix}  \mapsto
\begin{pmatrix}
1 & 0 & 2 & 3 & 0 \\
1 & 0 & 1 & -1 & 0 \\
0 & 1 & 0 & 0 & 0 \\
0 & 0 & 0 & 0 & 1
\end{pmatrix}
\end{equation*}
where we view the lines $\ell_1, \ell_2, \ell_3 \in \PP^1$
and $\ell''_1, \ell''_2, \ell''_3, \ell''_4, \ell''_5 \in \PP^4$ as being spanned by columns of matrices.

If $f: [n] \rightarrow [p]$ and $g: [p] \rightarrow [r]$ are injections, we {\bf do not} in general have the equality of 
maps $\nu_{g \circ f} = \nu_g \circ \nu_f$. 
For example, if $f: [3] \rightarrow [5]$ is as above and $g: [5] \rightarrow [6]$ is the injection
$g(1) = 3, g(2) = 4, g(3) = 2, g(4) = 6, g(5) = 1$ then
\begin{equation*}\nu_g \circ \nu_f:
\begin{pmatrix}
3 & 1 & 2 \\
-1 & 1 & 1 
\end{pmatrix}  \mapsto
\begin{pmatrix}
0 & 2 & 1 & 0 & 0 & 3 \\
0 & 1 & 1 & 0 & 0 & -1 \\
0 & 0 & 0 & 1 & 0 & 0 \\
1 & 0 & 0 & 0 & 0 & 0 \\
0 & 0 & 0 & 0 & 1 & 0
\end{pmatrix}
\end{equation*}
but
\begin{equation*}\nu_{g \circ f}:
\begin{pmatrix}
3 & 1 & 2 \\
-1 & 1 & 1 
\end{pmatrix}  \mapsto
\begin{pmatrix}
0 & 2 & 1 & 0 & 0 & 3 \\
0 & 1 & 1 & 0 & 0 & -1 \\
1 & 0 & 0 & 0 & 0 & 0 \\
0 & 0 & 0 & 1 & 0 & 0 \\
0 & 0 & 0 & 0 & 1 & 0
\end{pmatrix}.
\end{equation*}

In spite of this, we have

\noindent
{\bf Claim:}
{\em The assignment $[n] \mapsto H_d((\PP^{n-m-1})^n)$ is a finitely-generated $\FI$-module where an
injection $f: [n] \rightarrow [p]$ is sent to the induced homology map
$(\nu_f)_d: H_d((\PP^{n-m-1})^n) \rightarrow H_d(\PP^{p-m-1})^p$.}

To prove the Claim, consider injections $[n] \xrightarrow{f} [p] \xrightarrow{g} [r]$ as above.
We need to show the equality of maps $(\nu_{g \circ f})_d = (\nu_g)_d \circ (\nu_f)_d$ on the homology group
$H_d((\PP^{n-m-1})^n)$. 
To do this, we will establish the homotopy equivalence of the maps $\nu_{g \circ f}$ and $\nu_g \circ \nu_f$:
\begin{equation}
\label{target-homotopy-equivalence}
\nu_{g \circ f} \simeq \nu_g \circ \nu_f.
\end{equation}

The general linear group $GL_{r-m}(\CC)$ acts  on the projective space product 
 $(\PP^{r-m-1})^r$ by the rule $A \cdot (\ell'_1, \dots, \ell'_r) := (A \cdot \ell'_1, \dots, A \cdot \ell'_r)$
 for any $A \in GL_{r-m}(\CC)$.
 In terms of the matrices shown above, this corresponds to an action on rows.
 
 If we consider line tuples as matrices,
 it follows from the definition of $\nu$ that the image of any matrix under
  $\nu_{g \circ f}$ and $\nu_g \circ \nu_f$ coincides
 apart from some permutation of the bottom $r-n$ rows which only depends on $f$ and $g$.
 Rephrasing,
 there exists some $B \in GL_{r-m}(\CC)$ so that 
 $B \cdot \nu_{g \circ f} = \nu_g \circ \nu_f$ as functions on $(\PP^{n-m-1})^n$.
Since $GL_{r-m}(\CC)$ is path connected, we may choose some path from the identity to $B$
to induce the desired homotopy equivalence \eqref{target-homotopy-equivalence}.

We have shown that the assignment $[n] \mapsto H_d((\PP^{n-m-1})^n)$ is an $\FI$-module 
with respect to the maps $\nu_f$.  The well-known presentation
$H^{\bullet}((\PP^{n-m-1})^n) = \QQ[x_1, \dots, x_n]/\langle x_1^{n-m}, \dots, x_n^{n-m} \rangle$
where $x_i \leftrightarrow c_1(\ell_i^*)$ proves that this $\FI$-module is finitely generated.
This completes the proof of the Claim.

Observe that if $f: [n] \rightarrow [p]$ is an injection, we have $\nu_f(X_{n,n-m}) \subseteq X_{p,p-m}$.
In terms of matrices, this is the statement that $\nu_f$ takes full rank matrices to full rank matrices.
The
assignment $[n] \mapsto H_d(X_{n,n-m})$ is therefore a submodule of the finitely generated $\FI$-module
$[n] \mapsto H_d((\PP^{n-m-1})^n)$.
By Theorem~\ref{fi-is-noetherian}, the $\FI$-module 
$[n] \mapsto H^{\bullet}(X_{n,n-m})$ is itself finitely generated.
Another invocation of \cite[Thm. 1.13]{CEF} finishes the proof of Theorem~\ref{stability-theorem} (2).
\end{proof}

\section{Other Stability Results}
\label{Other}

\subsection{Higher dimensional spanning configurations}
The space $X_{n,k}$ of spanning line configurations can be generalized to higher dimensional subspaces.
Given a dimension $1 \leq d \leq k$ and $n \geq 1$ we have the product of 
Grassmannians
\begin{equation}
Gr(d,n,k) := Gr(d,k) \times \cdots \times Gr(d,k),
\end{equation}
where $Gr(d,k)$ is the Grassmannian of $d$-dimensional subspaces of $\CC^k$.
Thus $Gr(d,n,k)$ is the variety of all size $n$ configurations $(W_1, \dots, W_n)$ of $d$-planes in $\CC^k$.
Let $X_{d,n,k} \subseteq Gr(d,n,k)$ be the locus
\begin{equation}
X_{d,n,k} := \{ (W_1, \dots, W_n) \in Gr(d,n,k) \,:\, W_1 + \cdots + W_n = \CC^k \}
\end{equation}
of spanning $d$-plane configurations.
A point in $X_{2,3,3}$ is shown on the right of Figure~\ref{three-points}.

The Grassmann product $Gr(d,n,k)$ and
its subspace $X_{d,n,k}$ are closed under the $\symm_n$-action
\begin{equation*}
v.(W_1, \dots, W_n) := (W_{v(1)}, \dots, W_{v(n)}), \quad \quad v \in \symm_n.
\end{equation*}
The homology groups of $Gr(d,n,k)$ and $X_{d,n,k}$ inherit an action of $\symm_n$.

\begin{theorem}
\label{higher-dimensional-stability}
Let $p \geq 0$ be a fixed homological degree. 
\begin{enumerate}
\item  Fix $k \geq 1$.
The sequence of homology representations
$(H_p(X_{d,n,k}))_{n \geq 0}$
is uniformly representation stable.
\item  Fix $m \geq 0$.  
The sequence of homology representations
$(H_p(X_{d,n,dn-m}))_{n \geq 0}$
is uniformly representation stable. 
\end{enumerate}
\end{theorem}

\begin{proof}
This is very similar to our first proof of Theorem~\ref{stability-theorem}, so we only sketch the argument.
Let $N := dn$ and introduce $N$ variables $\xx_N = (x_1, x_2, \dots, x_N)$.  We break these variables
up into $n$ {\em batches} $\xx_N^{(1)} = (x_1, x_2, \dots, x_d), \dots, \xx_N^{(n)} = (x_{N-d+1}, \dots, x_{N-1}, x_N)$
of $d$ variables each.

Combining work of Rhoades \cite{RhoadesSpanning} and the Universal Coefficient Theorem, we have the following 
presentation of the cohomology ring of $X_{d,n,k}$.
\begin{equation}
H^{\bullet}(X_{d,n,k}) = (\QQ[x_1, x_2, \dots, x_N]/I)^{\symm_d \times \cdots \times \symm_d},
\end{equation}
where 
\begin{itemize}
\item the variables in the $i^{th}$ batch $\xx_N^{(i)}$ represent the Chern roots of the rank $d$ vector 
bundle $W_i^* \twoheadrightarrow X_{d,k,n}$,
\item the ideal $I$ is generated by the top $k$ elementary symmetric polynomials 
\begin{equation*}
e_N(\xx_N), e_{N-1}(\xx_N), \dots, e_{N-k+1}(\xx_N)
\end{equation*}
in the full variable set $\xx_N$ together with the $d$ complete homogeneous symmetric polynomials
\begin{equation*}
h_k(\xx_N^{(i)}), h_{k-1}(\xx_N^{(i)}), \dots, h_{k-d+1}(\xx_N^{(i)})
\end{equation*}
in the $i^{th}$ batch of variables $\xx_N^{(i)}$ for $i = 1, 2, \dots, n$, 
\item the $n$-fold product $\symm_d \times \cdots \times \symm_d$ permutes variables within batches, and
\item the superscript $(-)^{\symm_d \times \cdots \times \symm_d}$ denotes taking invariants.
\end{itemize}

Given this cohomology presentation, it is straightforward to check that (for fixed $d, k, m, p$) either of the assignments
\begin{equation}
[n] \mapsto H^{p}(X_{d,n,k}) \quad \text{or} \quad
[n] \mapsto H^{p}(X_{d,n,dn-m})
\end{equation}
is naturally a $\coFI$-module (here we permute and set batches of variables to zero wholesale).
Both of these are submodules of the $\coFI$-module
\begin{equation}
[n] \mapsto \QQ[x_1, \dots, x_{dn}]_{p/2}
\end{equation}
(where we interpret this as the zero module if $p$ is odd).
Representation stability follows from \cite[Thm. 1.13]{CEF}.
\end{proof}

Theorem~\ref{higher-dimensional-stability} shows that the sequences of homology representations 
in question are multiplicity stable.  However, the following problem remains open.

\begin{problem}
\label{cohomology-problem}
Calculate the graded isomorphism type of the $\symm_n$-module $H_p(X_{d,n,k})$.
\end{problem}

Problem~\ref{cohomology-problem} would be best solved by a tableau formula analogous to
Theorem~\ref{graded-frobenius}. The unresolved nature of Problem~\ref{cohomology-problem} illustrates
the leitmotif from the introduction:
 it can be easier to prove that a sequence of modules is representation stable than it is
to calculate their isomorphism types.

As with the cohomology presentation of $X_{n,k}$, a key component of the presentation of $H^{\bullet}(X_{d,n,k})$
in \cite{RhoadesSpanning} involved exhibiting a nonstandard affine paving 
\begin{equation}
\varnothing = Z_0 \subset Z_1 \subset \cdots \subset Z_m = Gr(d,n,k)
\end{equation}
of $Gr(d,n,k)$ which differs from the product Schubert variety paving such that 
 $X_{d,n,k} = Gr(d,n,k) - Z_i$ for some $i$.
Theorem~\ref{higher-dimensional-stability} could have also been proven using this affine paving alone and 
the known cohomology presentation of the Grassmann product
$Gr(d,n,k) = Gr(d,k) \times \cdots \times Gr(d,k)$.

\subsection{Anticommuting variables}
Let $\theta_1, \dots, \theta_n$ be anticommuting variables, i.e. $\theta_i \theta_j = - \theta_j \theta_i$
for $1 \leq i, j \leq n$.  Write $\QQ[\theta_1, \dots, \theta_n]$ for the $2^n$-dimensional
exterior algebra on these variables and write $\QQ[\theta_1, \dots, \theta_n]_d$ for its ${n \choose d}$-dimensional
subspace of homogeneous degree $d$.
The ring $\QQ[\theta_1, \dots, \theta_n]$ is a graded $\symm_n$-module by subscript permutation.

Just as in the commuting case, the assignment $[n] \mapsto \QQ[\theta_1, \dots, \theta_n]$ is an $\FI$-module.
However, this $\FI$-module is not finitely-generated. There are two ways to see this:
(1) the $\symm_n$-module
$\QQ[\theta_1, \dots, \theta_n]$ has a copy of the sign representation
$S^{(1^n)}$ in top degree, so that we do not have multiplicity stability and
(2) the dimension of $\QQ[x_1, \dots, x_n]$ is $2^n$, and if $V$ is any finitely-generated $\FI$-module the 
dimension sequence $\dim V(n)$ has eventually polynomial growth.
However, for fixed $d$ the assignment $[n] \mapsto \QQ[\theta_1, \dots, \theta_n]_d$ is a finitely-generated
$\FI$-module generated by the single element $\theta_1 \cdots \theta_d \in \QQ[\theta_1, \dots, \theta_d]_d$.

For $n, m, p \geq 0$ let $S(n,m,p)$ be the $\QQ$-algebra formed by taking the tensor product
of $m$ copies of the polynomial ring $\QQ[x_1, \dots, x_n]$ in commuting variables
and $p$ copies of the exterior algebra $\QQ[\theta_1, \dots, \theta_n]$ in anticommuting variables:
\begin{equation*}
S(n,m,p) := \overbrace{\QQ[x_1, \dots, x_n] \otimes \cdots \otimes \QQ[x_1, \dots, x_n]}^m \otimes 
\overbrace{\QQ[\theta_1, \dots, \theta_n] \otimes \cdots \otimes \QQ[\theta_1, \dots , \theta_n]}^p.
\end{equation*}
The ring $S(n,m,p)$ has a multigrading obtained by considering each tensor factor separately. 
For $\alpha = (\alpha_1, \dots, \alpha_m) \in (\ZZ_{\geq 0})^m$
and $\beta = (\beta_1, \dots, \beta_p) \in (\ZZ_{\geq 0})^p$ let 
$S(n,m,p)_{\alpha,\beta}$ be the piece of multidegree $(\alpha,\beta)$.


The rings $S(n,1,0)$ and $S(n,0,1)$ 
are the symmetric and exterior algebras
on $n$ generators. 
The ring $S(n,1,1)$ is called {\em superspace} and can be identified with the 
ring of polynomial valued differential forms on Eulidean $n$-space.

The symmetric group $\symm_n$ acts diagonally on the tensor factors of $S(n,m,p)$.
Orellana and Zabrocki recently proved \cite[Thm. 3.1]{OZ} 
a combinatorial interpretation of the $\symm_n$-isomorphism
type of the homogeneous piece $S(n,m,p)_{\alpha,\beta}$.
We consider this object for varying $n$.

\begin{proposition}
\label{s-is-stable}
Let $m,p \geq 0$ and let $\alpha \in (\ZZ_{\geq 0})^m$ and $\beta \in (\ZZ_{\geq 0})^p$ be multidegrees.
The sequence of $\symm_n$-modules
\begin{equation*}
S(n,m,p)_{\alpha,\beta} \quad \quad n = 1, 2, 3, \dots
\end{equation*}
is uniformly representation stable with respect to the natural embeddings 
\begin{equation*}
S(n,m,p)_{\alpha,\beta} \hookrightarrow S(n+1,m,p)_{\alpha,\beta}.
\end{equation*}
\end{proposition}

\begin{proof}
We have an $\FI$-module defined by $[n] \mapsto S(n,m,p)_{\alpha,\beta}$.  It is enough to show that 
this $\FI$-module is finitely generated.  
Any monomial in the $x$ and $\theta$ variables contained in 
the homogeneous component $S(n,m,p)_{\alpha,\beta}$ is $\symm_n$-conjugate (up to sign) to some monomial
in the $x$ and $\theta$ variables which lies in
$S(\alpha_1 + \cdots + \alpha_m + \beta_1 + \cdots + \beta_p, m, p)_{\alpha,\beta}$.
Since $m, p, \alpha$, and $\beta$ are fixed, there are only finitely many such conjugacy classes.
\end{proof}

Multihomogeneous quotients of $S(n,m,p)$ have received significant attention in algebraic combinatorics.
Bigraded quotients of superspace $S(n,1,1)$  generalizing the coinvariant algebra $R_n$
and specializing to the more general $\symm_n$-modules $R_{n,k}$
were studied by Rhoades and Wilson \cite{RWVandermonde}.
A trigraded quotient of $S(n,2,1)$ introduced by Zabrocki \cite{Zabrocki} conjecturally has 
Frobenius image equal to a symmetric function appearing in the Haglund-Remmel-Wilson
{\em Delta Conjecture}
in the theory of Macdonald polynomials
\cite{HRW}.  
We give this last example in more detail.

Let $S(n,p,m)^{\symm_n}_+$ be the subring of $S(n,p,m)$ 
consisting of $\symm_n$-invariant elements with vanishing constant term and
let $R(n,m,p)$ be 
the homogeneous quotient
\begin{equation}
R(n,m,p) := S(n,m,p)/\langle S(n,m,p)^{\symm_n}_+ \rangle.
\end{equation}
Write $R(n,m,p)_{\alpha,\beta}$ for the homogeneous piece of $R(n,m,p)$ of multidegree $(\alpha,\beta)$.

For any symmetric function $F$, the (primed) {\em delta operator} $\Delta_F': \Lambda \rightarrow \Lambda$
is a linear operator on the ring $\Lambda$ of symmetric functions over the ground field $\QQ(q,t)$.
For any partition $\mu$, let $\widetilde{H}_{\mu} \in \Lambda$ be the associated modified Macdonald
symmetric function.  The operator $\Delta'_F$ acts on $\widetilde{H}_{\mu}$ by
\begin{equation}
\Delta'_F: \widetilde{H}_{\mu} \mapsto F[B_{\mu}(q,t) - 1] \cdot \widetilde{H}_{\mu}.
\end{equation}
The eigenvalue $F[B_{\mu}(q,t) - 1]$ is obtained by filling every $(i,j)$ box of the Ferrers diagram of $\mu$ 
(aside from the northwest corner) with the monomial $q^i t^j$ and evaluating the symmetric function $F$
at these monomials. For example, if $\mu = (3,2)$ our filling is
\begin{small}
\begin{center}
\begin{young}
 $\cdot$ & $q$ & $q^2$ \cr
 $t$ & $qt$ 
\end{young}
\end{center}
\end{small}
so that $F[B_{\mu}(q,t) - 1] = F(q, q^2, t, qt)$.
The operator $\Delta'_F$ extends to an operator on all symmetric functions in $\Lambda$ by linearity.

The Delta Conjecture \cite{HRW} predicts the monomial expansion of $\Delta'_{e_{k-1}} e_n$ for any $k \leq n$.
It reads
\begin{equation}
\Delta'_{e_{k-1}} e_n = \Rise_{n,k}(\xx;q,t) = \Val_{n,k}(\xx;q,t)
\end{equation}
where $\Rise$ and $\Val$ are certain formal power series involving the infinite alphabet
$\xx = (x_1, x_2, \dots )$ and two additional parameters $q,t$; see \cite{HRW} for their definitions.
Zabrocki \cite{Zabrocki} verified by computer that
\begin{equation}
\label{zabrocki-conjecture}
\grFrob(R(n,2,1);q,t,z) = \sum_{k= 1}^{n} z^{n-j} \cdot \Delta'_{e_{k-1}} e_n
\end{equation}
(where $q, t$ track commuting degree and $z$ tracks anticommuting degree)
for $n \leq 6$ and conjectured that \eqref{zabrocki-conjecture} holds in general.

\begin{proposition}
\label{r-is-stable}
Let $m,p \geq 0$ and let $\alpha \in (\ZZ_{\geq 0})^m$ and $\beta \in (\ZZ_{\geq 0})^p$ be multidegrees.
The sequence of $\symm_n$-modules
\begin{equation*}
R(n,m,p)_{\alpha,\beta} \quad  \quad n = 1, 2, 3, \dots
\end{equation*}
is uniformly representation stable.
\end{proposition}

\begin{proof}
By \cite[Thm. 1.13]{CEF} it is enough to show $[n] \mapsto R(n,m,p)_{\alpha,\beta}$ is 
a finitely-generated $\coFI$-module.
By Theorem~\ref{fi-is-noetherian} 
and the proof of Proposition~\ref{s-is-stable}
we need only show that $[n] \mapsto R(n,m,p)$
is a quotient $\coFI$-module of $[n] \mapsto S(n,m,p)$; the finite generation of $R(n,m,p)_{\alpha,\beta}$
will follow by restricting to multidegree $(\alpha,\beta)$.

Let $f: [n] \rightarrow [r]$ be an injection and consider the associated map $f^*: S(r,m,p) \rightarrow S(n,m,p)$.
We have
$f^*(S(r,m,p)^{\symm_r}) \subseteq S(n,m,p)^{\symm_n}$ and therefore
$f^*(S(r,m,p)^{\symm_r}_+) \subseteq S(n,m,p)^{\symm_n}_+$  so that $f^*$ carries the generating set of the
ideal defining the quotient $R(r,m,p)$ into the generating set of the ideal defining $R(n,m,p)$.
\end{proof}

Rhoades and Wilson have an alternative conjectural model for $\Delta'_{e_{k-1}} e_n$ involving 
a superspace extension of the Vandermonde determinant \cite{RWVandermonde}.
Unlike the Zabrocki model  \eqref{zabrocki-conjecture}, its validity is known at $t = 0$.
Describing this model requires some notation.

Let $k \leq n$ and set $r := n-k$.  The {\em superspace Vandermonde} $\delta_{n,k} \in S(n,1,1)$
is
\begin{equation}
\delta_{n,k} := \varepsilon_n \cdot (x_1^{k-1} x_2^{k-1} \cdots x_r^{k-1} x_{r+1}^{k-1} x_{r+2}^{k-2} \cdots
x_{n-1}^1 x_n^0 \cdot \theta_1 \theta_2 \cdots \theta_r)
\end{equation}
where $\varepsilon_n := \sum_{w \in \symm_n} \sign(w) \cdot w \in \QQ[\symm_n]$.

The symmetric
partial derivative operator $\partial/\partial x_i$ acts naturally on $\QQ[x_1, \dots, x_n]$ for each $1 \leq i \leq n$.
For any $1 \leq i \leq n$ we have an antisymmetric 
partial derivative $\partial/\partial \theta_i$ acting on the exterior algebra
$\QQ[\theta_1, \dots, \theta_n]$ by
\begin{equation}
\partial/\partial \theta_i:  \theta_{j_1} \cdots \theta_{j_r} \mapsto \begin{cases}
(-1)^{s-1} \theta_{j_1} \cdots \theta_{j_{s-1}} \theta_{j_{s+1}} \cdots \theta_{j_r} & \text{if $i = j_s$} \\
0 & \text{if $i \notin \{j_1, \dots, j_r\}$}
\end{cases}
\end{equation}
for any $1 \leq j_1 < \cdots < j_r \leq n$ and linear extension; see \cite[Sec. 5]{RWVandermonde} for more details.
Given any commuting or anticommuting generator of $S(n,m,p)$ we have an action of the corresponding 
partial derivative on the relevant tensor factor of $S(n,m,p)$.

Let $m \geq 1$ and consider two tensor factors of $S(n,m,p)$ generated by commuting variables. 
To distinguish these tensor factors, we relabel so that these variables are $y_1, \dots, y_n$ and $z_1, \dots, z_n$.
For $j \geq 1$, the $j^{th}$ {\em polarization operator} from $y$-variables to $z$-variables is the linear
operator on $S(n,m,p)$ given by
\begin{equation}
\rho_{y \rightarrow z}^{(j)} := z_1 (\partial/\partial y_1)^j + z_2 (\partial/\partial y_2)^j + \cdots 
+ z_n (\partial/\partial y_n)^j.
\end{equation}
The operator $\rho_{y \rightarrow z}^{(j)}$ lowers $y$-degree by $j$ and raises $z$-degree by $1$.

Similarly, if $p \geq 1$ we may consider two tensor factors of $S(n,m,p)$ generated by anticommuting 
variables, which we relabel $\xi_1, \dots, \xi_n$ and $\tau_1, \dots, \tau_n$.
The {\em polarization operator} from $\xi$-variables to $\tau$-variables is the operator on $S(n,m,p)$
given by
\begin{equation}
\rho_{\xi \rightarrow \tau} := \tau_1 (\partial/\partial \xi_1) + \tau_2 (\partial/\partial \xi_2) + \cdots 
+ \tau_n (\partial/\partial \xi_n).
\end{equation}

Following \cite{RWVandermonde}, for $m, p \geq 1$ we let $V(n,m,p,k)$ be the smallest subspace of $S(n,m,p)$
which contains $\delta_{n,k}$ (in the `first' sets of commuting and anticommuting variables)
 and is closed under all partial derivative and polarization operators.
The space $V(n,m,p,k)$ is a multigraded $\symm_n$-module.

Rhoades and Wilson conjecture \cite[Conj. 6.3]{RWVandermonde} that 
\begin{equation}
\label{rw-conjecture} \text{coefficient of $z^{n-k}$ in }
\grFrob(V(n,2,1,k);q,t,z) =  \Delta'_{e_{k-1}} e_n
\end{equation}
where $q,t$ track commuting degree and $z$ tracks anticommuting degree.
The conjecture \eqref{rw-conjecture} is known at $t = 0$; by \cite[Thm. 4.2]{RWVandermonde} we have
\begin{equation}
\label{rw-theorem}
\text{coefficient of $z^{n-k}$ in }
\grFrob(V(n,1,1,k);q,z) =  \Delta'_{e_{k-1}} e_n \mid_{t = 0}.
\end{equation}
Proving
the corresponding specialization of the Zabrocki conjecture \eqref{zabrocki-conjecture}:
\begin{equation}
\grFrob(R(n,1,1);q,z) = \sum_{k = 1}^n z^{n-k} \cdot \Delta'_{e_{k-1}} e_n \mid_{t = 0}
\end{equation}
remains an open problem.

The $V$-modules enjoy the same stability property as the $R$-modules.

\begin{proposition}
\label{v-is-stable}
Let $m,p,k \geq 0$ and let $\alpha \in (\ZZ_{\geq 0})^m$ and $\beta \in (\ZZ_{\geq 0})^p$ be multidegrees.
The sequence of $\symm_n$-modules
\begin{equation*}
V(n,m,p,n-k)_{\alpha,\beta} \quad  \quad n = 1, 2, 3, \dots
\end{equation*}
is uniformly representation stable.
\end{proposition}

\begin{proof}
Thanks to Theorem~\ref{fi-is-noetherian} and Proposition~\ref{s-is-stable} it is enough to show that the natural
inclusion $S(n,m,p) \hookrightarrow S(n+1,m,p)$ sends
$V(n,m,p,n-k)$ into $V(n+1,m,p,n-k+1)$.  
The identity
\begin{equation}
\label{crucial-identity}
(\partial/\partial x_1) (\partial/\partial x_2) \cdots (\partial/ \partial x_n) \delta_{n+1,n-k+1} =
(n-k)^k \cdot (n-k)! \cdot \delta_{n,n-k}
\end{equation}
 may be verified directly from the definition of the superspace Vandermonde.
 Equation~\eqref{crucial-identity} implies that $\delta_{n,n-k} \in V(n+1,m,p,n-k+1)$ which in turn implies 
 $V(n,m,p,n-k) \subseteq V(n+1,m,p,n-k+1)$ as desired.
\end{proof}

It is known that $\Rise_{n,k}(\xx;q,t)$ is symmetric in $\xx$ and Schur-positive, but its Schur expansion
is unknown.
The expression $\Delta'_{e_{k-1}} e_n$ is symmetric, but is not known to be Schur-positive.
It is unknown whether $\Val_{n,k}(\xx;q,t)$ is symmetric.
Assuming these three formal power series are all Schur-positive symmetric functions,
Proposition~\ref{r-is-stable} and \eqref{zabrocki-conjecture} motivate the following problem.

\begin{problem}
\label{delta-combinatorial-problem}
Fix $i, j, k \geq 0$.  Prove combinatorially
that the sequences of $\symm_n$-modules $V_n, U_n, W_n$ whose Frobenius images 
are given by
\begin{align}
\Frob(V_n) &= \text{coefficient of $q^i t^j$ in $\Delta'_{e_{n-k-1}} e_n$} \\
\Frob(U_n) &= \text{coefficient of $q^i t^j$ in $\Rise_{n,n-k}(\xx;q,t)$} \\
\Frob(W_n) &= \text{coefficient of $q^i t^j$ in $\Val_{n,n-k}(\xx;q,t)$} 
\end{align}
are uniformly multiplicity stable.
\end{problem}

Our closing problem, which is likely very difficult, is a final illustration of our introductory leitmotif.

\begin{problem}
\label{multidegree-isomorphism-problem}
Let $n, m, p, k \geq 0$ and let $\alpha \in (\ZZ_{\geq 0})^m$ and $\beta \in (\ZZ_{\geq 0})^p$ be 
multidegrees. Calculate the isomorphism type of either of the $\symm_n$-modules
$R(n,m,p)_{\alpha,\beta}$ or $V(n,m,p,k)_{\alpha,\beta}$.
\end{problem}

\section{Acknowledgements}
\label{Acknowledgements}

E. Ramos was supported by NSF grant DMS-1704811. 
B. Rhoades was partially supported by NSF Grant DMS-1500838.
The authors are very grateful to Steven Sam and Ben Young for many helpful conversations.


\begin{thebibliography}{99}
 





\bibitem{Church} T. Church. Homological stability for configuration spaces of manifolds.
{\em Invent. Math.}, {\bf 188 (2)} (2012), 465--504.
 
 \bibitem{CEF} T. Church, J. S. Ellenberg, and B. Farb.
 FI-modules and stability for representations of symmetric groups.
 {\em Duke Math. J.} {\bf 164 (9)} (2015), 1833--1910.



\bibitem{HRW}  J. Haglund, J. Remmel, and A. T. Wilson.  The Delta Conjecture.  
{\it Trans. Amer. Math. Soc.}, {\bf 370} (2018), 4029--4057.

\bibitem{HRS}  J. Haglund, B. Rhoades, and M. Shimozono.  Ordered
set partitions, generalized coinvariant algebras, and the Delta Conjecture.
{\it Adv. Math.}, {\bf 329} (2018), 851--915.



\bibitem{OZ} R. Orellana and M. Zabrocki. 
A combinatorial model for the decomposition of multivariate polynomial rings as an $S_n$-module.
Preprint, 2019,
{\tt arXiv:1906.01125}.

\bibitem{PR}  B. Pawlowski and B. Rhoades.
A flag variety for the Delta Conjecture.
To appear, {\em Trans. Amer. Math. Soc.}, 2019.
{\tt arXiv:1711.08301}.



\bibitem{RhoadesSpanning} B. Rhoades. Spanning subspace configurations.
Preprint, 2019. {\tt arXiv:1903.07579}.

\bibitem{RWVandermonde} B. Rhoades and A. T. Wilson.
Vandermondes in superspace.
Preprint, 2019.
{\tt arXiv:1906.03315}.

\bibitem{Snowden} A. Snowden.
Syzygies of Segre embeddings and $\Delta$-modules.
{\em Duke Math. J.}, {\bf 162 (2)} (2013), 225--277.


\bibitem{WMultiset}  A. T. Wilson.  An extension of MacMahon's Equidistribution Theorem
to ordered multiset partitions. 
{\it Electron. J. Combin.}, {\bf 23 (1)} (2016), P1.5.

\bibitem{Zabrocki} M. Zabrocki. A module for the Delta conjecture.
Preprint, 2019. 
{\tt arXiv:1902.08966}.



  
\end{thebibliography}
\end{document}